\documentclass[reqno,12pt,a4paper]{amsart}

\usepackage{allrunes}
\usepackage[english]{babel}
\usepackage[T1]{fontenc}
\usepackage{mathtools,amssymb,verbatim,dsfont}
\usepackage{physics}
\usepackage{amsmath}
\usepackage{tikz}
\usepackage{tikz-3dplot}
\usepackage{pgfplots}
\usetikzlibrary{babel,3d}
\usepackage{bm}
\usepackage{float}
\usepackage{enumerate}
\usepackage{xfrac}
\usepackage{csquotes}
\usepackage[hidelinks]{hyperref}
\usepackage{caption}
\usepackage[backend=biber, style=numeric, sorting=nty]{biblatex}
\usepackage{etoolbox}
\usepackage{graphicx}
\usepackage{algorithm}
\usepackage{algpseudocode}
\usepackage{fancyhdr}
\usepackage{subcaption}
\usepackage{cleveref}
\usepackage{pdfpages}
\usepackage{xcolor}
\usepackage{multirow}
\makeatletter
\def\@footnotecolor{red}
\define@key{Hyp}{footnotecolor}{%
 \HyColor@HyperrefColor{#1}\@footnotecolor%
}
\patchcmd{\@footnotemark}{\hyper@linkstart{link}}{\hyper@linkstart{footnote}}{}{}
\makeatother
\hypersetup{footnotecolor=[RGB]{148,0,52},
linkcolor=[RGB]{148,0,52},
citecolor=[RGB]{148,0,52}}

\newtheorem{theorem}{Theorem}[section]
\newtheorem{lemma}[theorem]{Lemma}
\newtheorem{prop}[theorem]{Proposition}

\newtheorem{conjecture}[theorem]{Conjecture}

\theoremstyle{definition}
\newtheorem{definition}[theorem]{Definition}
\newtheorem{example}[theorem]{Example}

\theoremstyle{remark}
\newtheorem*{remark}{Remark}

\numberwithin{equation}{section}

\DeclareMathOperator*{\arginf}{arginf}

\DeclareMathOperator{\id}{id}

\DeclareMathOperator{\pr}{pr}

\DeclareMathOperator{\dist}{dist}
\DeclareMathOperator{\vol}{Vol}
\DeclareMathOperator{\dvol}{dVol}

\DeclareMathOperator{\eqdist}{\overset{d}{=}}

\DeclareMathOperator{\St}{St}

\DeclareMathOperator{\GL}{GL}

\title[Convergence of the EKF on Stiefel Manifolds]{On the Convergence of the Extended Kalman Filter on Stiefel Manifolds when Observing a Constant Particle with Measurement Errors}

\author[Figueras]{Jordi-Llu\'{i}s Figueras}

\author[Persson]{Aron Persson}

\author[Viitasaari]{Lauri Viitasaari}

\makeatletter
\let\newtitle\@title
\let\newauthor\@author
\makeatother

\pgfplotsset{compat=1.18}

\begin{document}

\begin{abstract}
In this paper we first introduce the setting of filtering on Stiefel manifolds. Then, assuming the underlying system process is constant, the convergence of the extended Kalman filter with Stiefel manifold-valued observations is proved. This corresponds to the case where one has measurement errors that needs to be filtered. Finally, some simulations are presented for a selected few Stiefel manifolds and the speed of convergence is studied.

\end{abstract}

\maketitle
\thispagestyle{empty}

\tableofcontents

\section{Introduction}

In the field of radiological MIMO-systems, it is important to be able to measure and make predictions of a state describing a system of antennas and receivers. One naturally models this state of $n$ antennas and $k$ receivers as an $n$ by $k$ matrix. Then, when observing the state, one measures the "orthogonal part" of the $n$ by $k$ matrix. This observation is done due to what is referred to as orthogonal frequency-division multiplexing in radiology allows for faster data transfers, see \cite{ergen2009mobile}. In other words, one considers noisy measurements that live on Stiefel manifolds, see e.g. \cites{hussien2014multi,pitaval2013joint,seddik2017multi}. Other contexts where Stiefel manifolds play an important role are computer vision, medical image analysis and machine learning, see \cites{cetingul2009intrinsic,cheema2024stiefelgen,kanamori2012non,lui2012advances,pham2008robust}. The setting of radiological MIMO systems motivates us to view the filtering problem as follows: Suppose $k$ receivers move around receiving signals and sending signals to $n$ antennas. The state of the system, abstractly referred to as the underlying process, moves around following some SDE with known parameters. Since the space of $n$ by $k$ matrices is linear, this SDE may be solved using classical methods. However, since the values of the measurements live in the space of orthogonal $n$ by $k$ matrices, one has Stiefel manifold-valued measurements. Thus, one may then ask: How to estimate $\mathbb{E}[X_t\mid Z]$, the state $X_t$ given $Z$ corresponding to a Stiefel manifold-valued measurement of the state $X_t$? 

Even if we assume that the underlying process is the solution to a linear SDE, this filtering problem is non-linear. In our case it is the measurement process that is explicitly non-linear. The measurements consist of mapping a point in $\mathbb{R}^n$ to the nearest point in a subspace of $\mathbb{R}^n$, see Equation \eqref{eq:filteringSDE} for the precise formulation. Such a mapping is only affine if the subspace itself is affine. Since the Stiefel manifolds are never affine subspaces, but closed manifolds, the projection map is strictly non-linear. This kind of setup of filtering equation is strictly different to the type of non-linear filtering one typically see, e.g. in \cite{bain2009fundamentals}, where the (potentially non-linear) measurements are linearly combined with the noise. In our case, the projected state of the underlying process is non-linearly mixed with the noise, so that the noisy measurement is also a point on a Stiefel manifold.

With this in mind, we introduced in \cite{figueras2025EKF} a version of the extended Kalman filter with Stiefel manifold-valued measurements. The classical extended Kalman filter, which was first introduced by Schmidt and others in \cites{mclean1962optimal,schmidt1966application,smith1962application}, is an extension of the original linear Kalman filter developed in \cites{kalman1960new,kalman1960contributions,kalman1961new} by K{\'a}lm{\'a}n and Bucy. In a nutshell, the Kalman filter gives the best estimate given an affine measurement which is weighed against a prediction (from the known affine dynamics). This is done mainly using that the propagation of the normal distributions through the linear filtering equation remain normal and easy formulas are derived. However, if the filtering equation is non-linear, as it is in our case, one has to settle with approximations.

Even if applications to MIMO systems and radio signals are important and they motivate why the Stiefel manifolds are reasonable objects on which projected data are considered, this paper has many other potential areas for application, such as problems in meteorology and biology that consider spherically-valued data. The spheres $\mathbb{S}^n$ are a special case of the Stiefel manifolds since $\mathbb{S}^n = \St_{n+1,1}(\mathbb{R})$. As such, the present article may also be useful to anyone interested in directional statistics on the spheres. The study of directional statistics has found itself embedded into quite significant number of disciplines, see \cites{ley2017modern,mardia2009directional}.

The area of directional statistics was popularized by Fischer in \cite{fisher1953dispersion} (see also \cite{mardia2025fisher} for a modern historical background) who considered spherically distributed data on rock fragments, in order to measure ancient Earth's magnetic field. Since then, directional statistics has been used in the following non-exhaustive list of areas of research; in ethology modelling animal movements \cites{mastrantonio2022modeling,schnute1992statistical}, in meteorology measuring changes in the directions of winds and waves \cites{bowers2000directional,hanson2009pacific,soukissian2014probabilistic}, and in medicine for analysing the temporal structure of time of birth \cite{demir2019application} or frequency of cancer \cite{karaibrahimoglu2021circular} or for self-evaluating health outcomes \cite{craig2010different}.

This article complements \cite{figueras2025EKF} where the extended Kalman filter with Stiefel manifolds is introduced, accompanied with simulation studies showing performance. In particular, we take a step towards theoretical guarantees of the convergence by showing that whenever the underlying process is a constant random variable, our algorithm converges. This corresponds to the case where instead of the true projection, one has measurement errors that has to be filtered out. 

The rest of the article is organized as follows. In Section \ref{sec:EKF} we introduce our setting and the results, while proofs are postponed to Section \ref{sec:proofs}. In Section \ref{sec:simulations} an abundance of promising simulations for a selection of Stiefel manifolds are provided. These simulations demonstrate for measurements with sufficiently small noise, the rate of convergence is very close to that of the classical linear Kalman filter.  On the other hand, no estimate for the rate of convergence can be easily attained if the measurement noise is large. This is due to the non-linear nature of the problem which isn't fully captured by the linearization process of the extended Kalman filter.

\section{The extended Kalman filter on Stiefel manifolds}
\label{sec:EKF}

\subsection{Background}
\label{sec:background}
Let $\mathbb{K}$ either denote the field $\mathbb{R}$ or $\mathbb{C}$ and let $\mathrm{M}_{n,k}(\mathbb{K})$ denote the $\mathbb{K}$-vector space of $n$ by $k$ matrices. For any $X\in \mathrm{M}_{n,k}(\mathbb{K})$ we shall denote the conjugate transpose of $X$ by $X^*=\overline{X}^T$. In the case that $\mathbb{K}=\mathbb{R}$ then the conjugate transpose is reduced to the real transpose, i.e. $X^*=X^T$.

Consider $X_t\in \mathrm{M}_{n,k}(\mathbb{K})$, the stochastic process satisfying the linear SDE
\begin{equation}
    \dd X_t = A X_t \dd t + \nu \dd B_t, \qquad X_0 \eqdist N(\mu_0,\sigma_0^2 I_{\mathrm{M}_{n,k}(\mathbb{K})})
    \label{eq:linSDEXt}
\end{equation}
where $B_t$ is the Brownian motion on $\mathrm{M}_{n,k}(\mathbb{K})$, $\nu \in \mathbb{K}$ and $A:\mathbb{K}^n\rightarrow \mathbb{K}^n$ a linear and anti-symmetric map. 

Suppose now that there is a noisy measurement process $Z_m$ which records the first $k$ directions of $X_t$, i.e.
\[
Z_m\in \St_{n,k} = \left\{ X\in \mathrm{M}_{n,k}(\mathbb{K}): X^*X= I_n \right\},
\]
the compact Stiefel manifold of $n$ by $k$ matrices.

Throughout we shall consider $\mathrm{M}_{n, n}(\mathbb{K})$ and $\St_{n,k}(\mathbb{K})$, equipped with the so-called canonical metric $g_c$. On $\mathrm{M}_{n, n}(\mathbb{K})$, let $g_c$ denote the Frobenius inner product on $\mathrm{M}_{n,n}(\mathbb{K})$ defined as
\[
g_c(V,W) = \trace V^* W
\]
for all $V,W\in \mathrm{M}_{n.n}(\mathbb{K})$. Since $\GL_{n,n}(\mathbb{K}) \subset \mathrm{M}_{n,n}(\mathbb{K})$ is open we get immediately the canonical metric on $\GL_{n,n}(\mathbb{K})$. Moreover, the canonical metric on $\St_{n,k}(\mathbb{K})$ at a point $X\in \St_{n,k}(\mathbb{K})$ is then defined as follows:
\[
g_c(X)(V,W)= \trace\left(V^* \left(I_n- \frac{1}{2}X X^*\right)W\right),
\]
for all $V,W \in T_{X}\St_{n,k}(\mathbb{K})$. These choices of metrics turns the following submersion
\[
\iota :  \GL_{n,n}(\mathbb{K}) \rightarrow \St_{n,k} (\mathbb{K}), \qquad \iota (X) = \pr(X \cdot I_{n,k})
\]
into an isometry.

Given any maximal rank $n$ by $k$ matrix, i.e. $X\in \GL_{n,k}(\mathbb{K})$, there is a unique closest point in $\St_{n,k}(\mathbb{K})$ if the induced norm from the metric is invariant under $O(n)$ (resp. $U(n)$) action (This was originally shown by Fan and Hoffman \cite{fan1955some}, see also \cite[Theorem 8.4]{higham2008functions}, and  \cite{laszkiewicz2006approximation}). Our case applies because the canonical metric $g_c$ is compatible with the Lie group action. This closest point may be computed by the polar decomposition: given $X\in \GL_{n,k}(\mathbb{K})$ then there is a unique orthogonal matrix, $U\in  \St_{n,k}(\mathbb{K})$, and unique symmetric matrix, $S\in \mathrm{M}_{k,k}(\mathbb{K})$, such that $X=US$. The symmetric matrix $S$ may be computed using the relation
\[
S= (X^* X)^{1/2}
\]
which is invertible since $X$ is of full rank. Hence the projection map 
\[
{\pr:\GL_{n,k}(\mathbb{K}) \rightarrow \St_{n,k}(\mathbb{K})}
\]
may be explicitly computed as
\begin{equation}
\pr(X) = X(X^*X)^{-1/2}.
\label{eq:projection}
\end{equation}

Thanks to Definition \eqref{eq:projection} the following lemma is immediately deducible.

\begin{lemma}
\label{lem:projectioncommuteswithaction}
    Let $X\in \GL_{n,k}(\mathbb{R})$ (resp. $\GL_{n,k}(\mathbb{C})$) and let $\phi \in O(n)$ (resp. $U(n)$) be arbitrary. Then it holds that 
    \[
    \pr(\phi X)= \phi \pr(X).
    \]
    
\end{lemma}

Hence, the projection map commutes with left-Lie group action and thus respects the homogeneous structure of the Stiefel manifolds. This lemma allows us to show that the projection map \eqref{eq:projection} commutes with the expectation, given the covariance is isotropic.

\begin{prop}
\label{prop:projectionstiefelcommuteswithmean}
    Given $X\eqdist N(\mu, \sigma^2 I_{\mathrm{M}_{n,k}(\mathbb{K})})$ with $ \mu \in \GL_{n,k}(\mathbb{K})$. Then, 
    \[
    \pr(\mu)=\mathbb{E}[\pr(X)]:= \arginf_{p\in \St_{n,k}(\mathbb{K})} \int_{\St_{n,k}(\mathbb{K})}  \dist^2(p,y) p_{\pr(X)}(y)  d\vol(y).
    \]
\end{prop}

The proof is found in Section \ref{sec:proofs}. The notion of projected normal random variables naturally extend from the definition of projected normal random variables on the spheres.

\begin{definition}
    Let $X\eqdist N(\mu,\sigma^2 I_{\mathrm{M}_{n,k}(\mathbb{K})})$ be normal random variable such that $\mathbb{P}(X\in \GL_{n,k}(\mathbb{K}))=1$. Then the random variable $\pr(X)$ is called a \textit{projected normal random variable}.
\end{definition}

Hence, Proposition \ref{prop:projectionstiefelcommuteswithmean} tells us that projected normal random variables, $\pr(X)$, have the same expectation as $X$, up to a factor of a symmetric $k$ by $k$ matrix, if the covariance of $X$ is isotropic.

\subsection{The filtering problem}

We shall model the filtering problem as measuring the SDE in \eqref{eq:linSDEXt} as follows:
\begin{equation}
    \begin{cases}
        \dd X_t = A X_t \dd t + \nu \dd B_t, & X_0 \eqdist N(\mu_0,\sigma_0^2 I_{\mathrm{M}_{n,k}(\mathbb{K})})\\
        Z_m= \pr\left( \pr(X_{t_m})+\varepsilon_m \right)
    \end{cases},
    \label{eq:filteringSDE}
\end{equation}
where $\pr$ is the projection onto $\St_{n,k}(\mathbb{K})$ defined in Equation \eqref{eq:projection}, $\varepsilon_m\in \mathrm{M}_{n,k}(\mathbb{K})$ are independent identically distributed normal random variables, following the distribution of $N(0,\xi^2 I_{\mathrm{M}_{n,k}(\mathbb{K})})$, and $0<t_1 <\dots <t_m$ is a discrete increasing sequence. Furthermore, we shall consider the measurement noise $\varepsilon_m$ to be independent of $X_t$ for all $t\geq 0$. The measurement process can be interpreted as follows; first consider the normal random variable with mean $\pr(X_{t_m})$ and with covariance matrix $\xi^2 I_{\mathrm{M}_{n,k}(\mathbb{K})}$. Then this normal random variable is projected onto $\St_{n,k}(\mathbb{K})$. It follows that the conditional probability distribution of $Z_m$ given $X_{t_m}$, denoted by $p_{z_m\mid X_{t_m}}$, follows the projected normal distribution.

\begin{remark}
Another plausible measurement process could be 
\[
\Tilde{Z}_m =\pr(X_{t_m}+ \varepsilon)
\]
as this would make $\Tilde{Z}_m$ itself a projected normal random variable, whilst $Z_m$ in Equation \eqref{eq:filteringSDE} does not follow the projected normal distribution since $Z_m$ may be computed by using the law of total probability
\[
p_{z_m} (z) = \int_{\St_{n,k}(\mathbb{K})} p_{Z_m\mid X_{t_m}=x}(z) p_{X_{t_m}}(x) \dvol_{\St_{n,k}(\mathbb{K})}(x),
\]
and this will never follow the projected normal distribution unless either $p_{Z_m\mid X_{t_m}=x}(z)$ or $ p_{X_{t_m}}(x)$ follow the point mass distribution. However, the measurement process in Equation \eqref{eq:filteringSDE} is the correct choice out of these two. The covariance of $\Tilde{Z}_m$ given $X_{t_m}$ depends on the realization of $X_{t_m}$. More precisely, it depends on $\norm{X_{t_m}}$. On the other hand the covariance of the measurement process given by $Z_m$ does not depend on the realization of $X_{t_m}$. 
\end{remark}

We propose the following concept of maximal scalar variance of a manifold that, as far as we know, is only previously used without label nor name in \cite{figueras2024parameter}. The maximal scalar variance of a manifold describes the maximally attainable scalar variance for a distribution on the underlying manifold. On $\mathbb{R}^n$, there is no such maximally attainable scalar variance as arbitrarily large scalar variances for a distribution on $\mathbb{R}^n$ are possible. On compact manifolds having suitable geometric properties however, there is typically a finite upper bound for the maximal attainable scalar variance of a distribution. This is useful in order to approximate the scalar variance on manifolds when a normal random variable is projected onto the manifold.

\begin{definition}
    Let $M$ be a compact Riemannian manifold with associated distance function $\dist$ and Riemannian volume form $\vol$. Then \textit{the maximal scalar variance on $M$} is the value 
    \[
     \mathcal{M}(M)= \frac{1}{\dim(M)\vol(M)} \sup_{p\in M}\int_{M} \dist^2(p,y)\dvol(y).
    \]
\end{definition}

\begin{remark}
    Consider a Lie group $G$. Note that if $M$ is a Riemannian G-symmetric space then the Riemannian distance function, $\dist$, is invariant under G-action, since by definition the action is isometric. Moreover, if $G\subseteq O(n)$ (resp. $U(n)$) then a change of coordinates done by action from any element from $G$ has Jacobian determinant $1$. Therefore, for any such space $M$, the integral
    \[
    \int_{M} \dist^2(p,y)\dvol_M(y)
    \]
    is the same irrespective of the choice of $p\in M$. Hence, for such spaces the maximal scalar variance is simplified to
    \[
        \mathcal{M}(M)=\frac{1}{\dim(M)\vol(M)} \int_{M} \dist^2(p,y)\dvol_M(y).
    \]
    for any $p$. Note that in particular that this applies to the spheres $\mathbb{S}^n$ as they are $O(n)$-symmetric spaces.
\end{remark}

We shall give some explicit computations of the maximal scalar variance on the the spheres as examples. First we shall look at the unit circle $\mathbb{S}^1 \cong \St_{2,1}(\mathbb{R})$, which is also the easiest case.

\begin{example}[Maximal scalar variance for $\mathbb{S}^1$]
    Parametrise $\mathbb{S}^1$ over $(0,2\pi)$ by $\theta \longmapsto (\cos(\theta),\sin(\theta))$ and take $p= \pi$ then $\dist(p,y)^2= \abs{\theta-\pi}^2$. Therefore it is immediate that
    \[
    \mathcal{M}(\mathbb{S}^1) = \frac{1}{2\pi} \int_0^{2\pi} \abs{\theta-\pi}^2 \dd \theta = \frac{\pi^2}{3}.
    \]
\end{example}

Next we shall consider the maximal scalar variance for $\mathbb{S}^n$ with $n\geq 2$ even.

\begin{example}[Maximal scalar variance for the even dimensional spheres]
\label{ex:maxvareven}
Let $n\geq 2$ be even. Using spherical coordinates for $\mathbb{S}^n$
\[
\begin{pmatrix}
    x_1\\x_2\\ x_3\\ \vdots \\x_{n-1} \\x_n \\x_{n+1}
\end{pmatrix} = \begin{pmatrix}
\sin(\theta_1)\sin(\theta_2)\dots \sin(\theta_{n-2}) \cos(\theta_{n-1})\sin(\phi)\\
\sin(\theta_1)\sin(\theta_2)\dots \sin(\theta_{n-2}) \sin(\theta_{n-1}) \sin(\phi)\\
\sin(\theta_1)\sin(\theta_2)\dots \cos(\theta_{n-2}) \sin(\phi)\\
\vdots\\
\sin(\theta_1)\cos(\theta_2)\sin(\phi)\\
\cos(\theta_1)\sin(\phi)\\
\cos(\phi)
\end{pmatrix},
\]
where $\theta_{n-1}\in [0,2\pi)$ and all other variables are in $[0,\pi]$. Moreover note that the volume form  in these coordinates is given by
\[
\dvol_{\mathbb{S}^n} = \sin^{n-1}(\phi) \sin^{n-2}(\theta_1) \sin^{n-3}(\theta_2) \dots \sin(\theta_{n-2}) \dd \phi \dd \theta_1 \dots \dd \theta_{n-1}.
\]
It is classical that we have the primitive function
\begin{multline*}
\psi(\phi):= \int \sin^{n-1}(\phi) \dd \phi = -\sum_{k=0}^{\frac{n-2}{2}} \sin^{n-2-2k}(\phi) \cos(\phi) \frac{(n-3-2k)!!}{(n-1)!!} \\
\cdot \frac{(n-2)!!}{(n-2-2k)!!},    
\end{multline*}
where $!!$ denotes the double factorial and by convention it holds that 
\[
1!!=0!!=(-1)!!=1.
\]
Choosing $p= (0,0,\dots,0,1)^T$, then $\dist(x,p)= \phi$ in the spherical coordinates. Using integration by parts it holds that
\begin{multline*}
\mathcal{M}(\mathbb{S}^n) = \frac{1}{n \vol(\mathbb{S}^n)} \int_{\mathbb{S}^n} \phi^2 \dvol_{\mathbb{S}^n}(\theta_1,\dots, \theta_{n-1}, \phi)\\
= \frac{1}{n\vol(\mathbb{S}^n)} \frac{(n-1)!!}{2(n-2)!!} \vol(\mathbb{S}^n) \int_0^\pi \phi^2 \sin^{n-1}(\phi) \dd \phi\\
= \frac{(n-1)!!}{2n(n-2!!} \left(\left[ \phi^2 \psi(\phi) \right]^{\pi}_{0} - \int_{0}^\pi 2\phi \psi(\phi) \dd \phi \right)\\
= \frac{1}{2n} \left(\pi^2 - 4\sum_{k=0}^{\frac{n-2}{2}} \frac{1}{(2k+1)^2} \right).
\end{multline*}
    In the case that $n=2$, the maximal scalar variance is    
    \[
    \mathcal{M}(\mathbb{S}^2)= \frac{\pi^2 -4}{4}.
    \]
\end{example}

\begin{example}[Maximal scalar variance for the odd dimensional spheres]
    Let $n\geq 3$ be odd. Consider the same choice of spherical coordinates as in Example \ref{ex:maxvareven} and with the same choice of $p$. We have the following primitive function
    \begin{multline*}
        \psi(\phi) = \int \sin^{n-1}(\phi) \dd \phi \\ = - \sum_{k=0}^{\frac{n-3}{2}}\sin^{n-2-2k}(\phi) \cos(\phi) \frac{(n-3-2k)!!}{(n-1)!!} \frac{(n-2)!!}{(n-2-2k)!!}\\ + \phi \frac{(n-2)!!}{(n-1)!!}.
    \end{multline*}

    Then using integration by parts
    \begin{align*}
        \mathcal{M}(\mathbb{S}^n) &= \frac{1}{n \vol(\mathbb{S}^n)} \int_{\mathbb{S}^n} \phi^2 \dvol_{\mathbb{S}^n}(\theta_1,\dots, \theta_{n-1}, \phi)\\
        &=\frac{1}{n \vol(\mathbb{S}^n)} \frac{(n-1)!! \vol(\mathbb{S}^n)}{\pi (n-2)!!} \int_0^\pi \phi^2 \sin^{n-1}(\phi) \dd \phi\\
        &=\frac{(n-1)!!}{n\pi (n-2)!!}\left(\left[\phi^2 \psi(\phi) \right]_0^\pi- \int_0^\pi 2 \phi \psi(\phi) \dd \phi  \right)\\
        &=\frac{1}{n} \left( \frac{\pi^2}{3} - 2\sum_{k=1}^{\frac{n-1}{2}} \frac{1}{(2k)^2} \right).
    \end{align*}
    Then for the special case that $n=3$ it follows that
    \[
    \mathcal{M}(\mathbb{S}^3)= \frac{1}{3}\left(\frac{\pi^2}{3}-\frac{1}{2} \right).
    \]
\end{example}

The intrinsic scalar variance for the projected normal random variable $\pr(X)$, where
\[
X\eqdist N(\mu,v^2I_{\mathrm{M}_{n,k}(\mathbb{K})})
\]
with $\mu \in \St_{n,k}(\mathbb{K})$, is given by
\begin{equation}
    \eta(v^2)= \frac{1}{\dim(\St_{n,k}(\mathbb{K}))} \int_{\St_{n,k}(\mathbb{K})} \dist^2(x,\mu) p_{\pr(X)}(x) \dvol_{\St_{n,k}(\mathbb{K})}(x).
    \label{eq:etaexact}
\end{equation}
Since $\eta$ is difficult to compute explicitly, we shall use the maximal scalar variance in a two-point Pad\'e approximation, akin to \cite{abouir2003multivariate}, to approximate the projected normal (intrinsic) scalar variance. In other words, for 
\[
X\eqdist N(\mu,v^2I_{\mathrm{M}_{n,k}(\mathbb{K})})
\]
with $\mu \in \St_{n,k}(\mathbb{K})$ it holds that
\begin{equation}
    \eta(v^2) \approx \hat{\eta}(v^2):= \frac{v^2 \mathcal{M}(\St_{n,k}(\mathbb{K}))}{\mathcal{M}(\St_{n,k}(\mathbb{K})) + v^2}.   
    \label{eq:padeapproxeta}
\end{equation}

The approximation in \eqref{eq:padeapproxeta} is based on Conjecture \ref{conj:smallvarianceforstiefel} below justifying that $\hat{\eta}(v^2)$ is a good approximation for the true intrinsic scalar variance. This is indeed the case for  $\mathbb{S}^2$, see \cite{figueras2024parameter}, where it was shown that mapping of $v^2$ to the intrinsic scalar variance of the projected Normal on the manifold $\mathbb{S}^2$ is actually a bijective mapping. In the case of a more general manifold, simulations show that the conjecture seems to be true and that the mapping between $v^2$ to the intrinsic scalar variance of a projected normal random variable seems to be bijective in this case as well. However,  our algorithm is based on using $\hat{\eta}(v^2)$, instead of moment matching the true variance. Since $\hat{\eta}$ is clearly injective, this allows us to obtain convergence even without proving bijectivity results as in \cite{figueras2024parameter} which would be a very challenging task in the general case $\St_{n,k}(\mathbb{K})$.

\begin{conjecture}
\label{conj:smallvarianceforstiefel}
 Let $X\eqdist N(\mu,v^2I_{\mathrm{M}_{n,k}(\mathbb{K})})$ with $\mu \in \St_{n,k}(\mathbb{K})$. Then,
 \[
 \frac{1}{\dim{\St_{n,k}(\mathbb{K})}}\int_{\St_{n,k}(\mathbb{K})} \dist^2(\mu,x) p_{\pr(X)}(x) \dvol_{\St_{n,k}(\mathbb{K})}(x) = v^2 + \mathcal{O}(v^3)
 \]
 for small $v$.
\end{conjecture}

% \begin{remark}
%     We use this Pad\'e approximation instead of moment matching the true variance of the projected normal, i.e.
%     \[
%     \frac{\trace \Cov(\mathbb{E}[X_{t_i}\mid Z_1=z_1,\dots, Z_i=z_i])}{\dim(\St_{n,k}(\mathbb{K}))}
%     \]
%     as it is very difficult to attain. Moreover, it is still an unsolved problem whether there is a bijection
%     \[
%     \frac{\trace \Cov(X)}{ nk}  \longmapsto \frac{\trace \Cov(\pr(X))}{\dim(\St_{n,k}(\mathbb{K})))}
%     \]
%     for $X\eqdist N(\mu,\Sigma) \in \mathrm{M}_{n,k}(\mathbb{K})$ akin to the case of $\mathbb{S}^2$ see \cite{figueras2024parameter}.
% \end{remark}

In \cite{figueras2025EKF} we introduced a numerical scheme, see Algorithm \ref{alg:kalmansphere} below, which approximates a solution to the filtering problem in \eqref{eq:filteringSDE}.

\begin{algorithm}[ht]
\caption{One step of the extended Kalman filter on $\St_{n,k}(\mathbb{K})$ after time $t$.}
\label{alg:kalmansphere}
\begin{algorithmic}[1]
    \State Given prior $x_0 \eqdist N(\mu_0,\sigma_0^2 I_{\mathrm{M}_{n,k}(\mathbb{K})})$ in $\mathrm{M}_{n,k} (\mathbb{K})$ with $\mu_0 \in \St_{n,k}(\mathbb{K})$ and a measurement $z_1\in \St_{n,k}(\mathbb{K}) $.
    \State $F_t = \exp_M(t A)$
    \State $\mu_{\text{pred.}} = F_t \cdot \mu_0$
    \State $P_{\text{pred.}} = \hat{\eta}(\sigma_0^2 + t\nu^2)= \frac{\sigma_0^2 + t\nu^2}{\mathcal{M}(\St_{n,k}(\mathbb{K}))+\sigma_0^2 + t\nu^2}\mathcal{M}(\St_{n,k}(\mathbb{K}))$
    \State $y= \log_{\mu_{\text{pred.}}}(z_1)$
   %\State $S \approx \eta_{\pr}(\sigma_0^2+t\nu^2+\xi^2)\approx \frac{\sigma_0^2+t\nu^2+\xi^2}{\pi^2-4+4(\sigma_0^2+t\nu^2+\xi^2)} (\pi^2-4$)
    \State $K=  \frac{\sigma_0^2 + t\nu^2 }{\sigma_0^2 + t\nu^2 + \xi^2} $ %\frac{\eta_{\pr}(\sigma_0^2 + t\nu^2)}{\eta_{\pr}(\sigma_0^2+t\nu^2+\xi^2)}  \approx \frac{(\sigma_0^2 + t\nu^2)(\pi^2-4+4(\sigma_0^2+t\nu^2+\xi^2))}{(\pi^2-4+4(\sigma_0^2 + t\nu^2))(\sigma_0^2+t\nu^2+\xi^2)}$
    \State $\mu^{K} = \exp_{\mu_{\text{pred.}}}(K y)$
    \State $P^{K} = (1-K)P_{\text{pred.}}$
    \State The estimated Projected distribution is now $PrN(\mu^{K},P^{K} \id_{T_{\mu^{K}} \St_{n,k}(\mathbb{K})})$ with corresponding Normal distribution 
    \begin{multline*}        
    N\Big(\mu_{K} , \frac{P^{K} \mathcal{M}(\St_{n,k}(\mathbb{K}))}{\mathcal{M}(\St_{n,k}(\mathbb{K}))-P^{K}}I_{\mathrm{M}_{n,k} (\mathbb{K})} \Big)\\= N\Big(\mu_{K},\resizebox{8cm}{!}{$\frac{\xi^2(\sigma_0^2 + t\nu^2)\mathcal{M}(\St_{n,k}(\mathbb{K}))}{(\xi^2+ \sigma_0^2+t\nu^2)(\mathcal{M}(\St_{n,k}(\mathbb{K}))+\sigma_0^2+t\nu^2) - \xi^2(\sigma_0^2 + t\nu^2)}$} I_{\mathrm{M}_{n,k} (\mathbb{K})}\Big)
    \end{multline*}

\end{algorithmic}
\end{algorithm}

The algorithm works as follows: Firstly, consider the filtering problem defined by the system in \eqref{eq:filteringSDE}. Since the evolution of the particle is described by a linear SDE on $\mathbb{R}^n$, we can solve it exactly using the evolution function given in step 2. In step 3, the predicted average $F_t \cdot \mu_0$ is an element $\St_{n,k}$ by assumption on $\mu_0$ and since $A$ is anti-symmetric. This predicted point is then also an unbiased estimate of the true mean of $\pr(X_t)$ by Proposition \ref{prop:projectionstiefelcommuteswithmean}. The predicted scalar variance of the projected distribution is then approximated by Equation \eqref{eq:padeapproxeta} in the 4th step. In step 5 we compute the innovation process using the geometry of $\St_{n,k}(\mathbb{K})$. Next, in the 6th step we compute the Kalman gain $K$, which is a measurement on how much one should trust the prediction vs the measurement. In Step 7, the Kalman gain multiplied with the innovation process gives the direction and speed one should travel, along the geodesic defined by said speed and direction, starting at the predicted mean to end up at the filtered mean. Lastly we compute the Kalman variance estimate in step 8, which is then approximated as a projected normal distribution which is mapped back to a corresponding Normal distribution on $\mathrm{M}_{n,k}(\mathbb{K})$ through the inverse of the map in Equation \eqref{eq:padeapproxeta}. See \cite{figueras2025EKF} for a more detailed description how the filter works. 

Now we may state the main theorem which implies that if the system process eventually reaches a stable state, the Extended Kalman filter given in Algorithm \ref{alg:kalmansphere} converges and the filtered mean converges to the true state in probability.

\begin{theorem}
\label{thm:convofextkalman}
Let $X_t$ follow the SDE given in \eqref{eq:linSDEXt}, with $A=0$ and $\nu=0$, thus $X_t=X_0$ is the constant process  where $X_0\eqdist N(\mu_0,\sigma^2_0 I_{\mathrm{M}_{n,k}})$, $\mu_0\in \St_{n,k}(\mathbb{K})$. Furthermore, let $z_1,\dots,z_k$ be realizations of the measurements in \eqref{eq:filteringSDE}
and let $\mu^{K}_m$ be the filtered mean estimate $P^K_m$ is the Kalman variance estimate after $m$ iterations of Algorithm \ref{alg:kalmansphere} applied to the measurements $z_1,\dots,z_m$. Then,
\[
    \dist(\mu^K_m,X_0)\rightarrow 0
\]
in probability as $m\to \infty$. Moreover,
\[
P^{K}_m \to 0
\]
as $m\to \infty$. 
\end{theorem}

Note that if Conjecture \ref{conj:smallvarianceforstiefel} holds true, Theorem \ref{thm:convofextkalman} tells us that $P^K_m$ will be a better approximation for large $m$ given that $\xi^2$ is small enough, and the true scalar variance of $\mathbb{E}[X_t\mid Z_{t_m}]$ also goes to zero.

\section{Proofs}
\label{sec:proofs}

\subsection{Proof of Proposition \ref{prop:projectionstiefelcommuteswithmean}}

Take any $\phi \in O(n)$ (resp. $U(n)$) such that 
\[
\phi\mu=\mu\in \GL_{n,k}(\mathbb{R}) \left(\text{resp. } \GL_{n,k}(\mathbb{C})\right)
\]
then for $X\eqdist N(\mu,\sigma^2 I_{\mathrm{M}_{n,k}(\mathbb{R})})$ (resp. $X\eqdist N(\mu,\sigma^2 I_{\mathrm{M}_{n,k}(\mathbb{C})})$) it holds that 
\[
\phi X\eqdist X.
\]
Therefore the probability density function for $\pr(X)$ is symmetric around $\mu$ since
\[
\phi\pr(X) = \pr(\phi  X) \eqdist \pr(X)
\]
by Lemma \ref{lem:projectioncommuteswithaction}. Note that
\[
    \dist^2(p,y) = \norm{\log_p(y)}^2
\]
and therefore as long as $p$ is sufficiently close to $y$ it holds that
\[
\frac{\dd }{\dd p} \dist^2(p,y) = 2\log_p(y).
\]
Therefore, 
\begin{multline*}
\frac{\dd}{\dd \mu} \int_{\St_{n,k}} \dist^2(\mu,y) p_{\pr(X)}(y) \vol(y) \\
=\int_{\St_{n,k}} \frac{\dd}{\dd \mu}  \dist^2(\mu,y) p_{\pr(X)}(y) \vol(y)\\
=\int_{\St_{n,k}}   2\log_\mu(y) p_{\pr(X)}(y) \vol(y) =0
\end{multline*}
by symmetry since $\log_\mu$ is anti-symmetric around $\mu$. That this is indeed a global minimum follows from observing that
\[
p_{\pr(X)} (\mu) > p_{\pr(X)}(-\mu)
\]
and that for any geodesic ball $B_r(\mu) \subseteq \St_{n,k}$ of radius $r$ it holds that
\[
\mathbb{P}(\pr(X)\in B_r(\mu)) \geq \mathbb{P}(\pr(X) \in B_r(-\mu)).
\]

\qed

\subsection{Proof of Theorem \ref{thm:convofextkalman}}
    In order to finish the proof we shall need two additional lemmas. But before that, we shall make a few observations. Let $\hat{\eta}:[0,\infty)\rightarrow \mathbb{R}$ denote the map defined in Equation \eqref{eq:padeapproxeta}.
    Therefore we may write that $P^K_1= (1-K_1)\hat{\eta}(\sigma_0^2)$, since $\nu= 0$. However it holds that
    \[
    \hat{\eta}^{-1}(r)= \frac{r \mathcal{M}(\St_{n,k}(\mathbb{K}))}{\mathcal{M}(\St_{n,k}(\mathbb{K}))- r}
    \]
    and thus
    \[
    \hat{\eta}^{-1}((1-K)\hat{\eta}(\sigma_0^2)) = \frac{(1-K)\hat{\eta}(\sigma_0^2) \mathcal{M}(\St_{n,k}(\mathbb{K}))}{\mathcal{M}(\St_{n,k}(\mathbb{K}))-(1-K)\hat{\eta}(\sigma_0^2)}.
    \]
    Since $(1-K)\leq 1$ it holds that
    \[
            \frac{(1-K)\hat{\eta}(\sigma_0^2) \mathcal{M}(\St_{n,k}(\mathbb{K}))}{\mathcal{M}(\St_{n,k}(\mathbb{K}))-(1-K)\hat{\eta}(\sigma_0^2)} \leq (1-K)\frac{\hat{\eta}(\sigma_0^2) \mathcal{M}(\St_{n,k}(\mathbb{K}))}{\mathcal{M}(\St_{n,k}(\mathbb{K}))-\hat{\eta}(\sigma_0^2)}
    \]
    and one can conclude that
    \[
    \hat{\eta}^{-1}((1-K)\hat{\eta}(\sigma^2_0))\leq (1-K) \sigma_0^{2}.
    \]
    Observe that the right hand side of the above inequality is exactly the variance of the filtered process after one step in the linear Kalman filter on $\mathbb{R}$. Therefore, $\hat{\eta}^{-1}(P^K_i)$ is dominated by the classical linear Kalman filter variance for observing the constant process on $\mathbb{R}$ with linear measurements and variance $\xi^2$. Since the classical variance estimate for the classical linear Kalman filter goes to $0$, it holds that
    \begin{equation}
    \label{eq:limitvarianceestimate}
        \lim_{i\to \infty}P^K_i  = 0.
    \end{equation}

    \begin{remark}
    Note that this makes the Kalman scalar variance estimator $P^K_i$ an optimistic estimator, the opposite of conservative.   
    \end{remark}

\begin{lemma}
\label{lem:Kalmangains}
    Let $K_m$ denote the $m$:th step Kalman gain following Algorithm \ref{alg:kalmansphere} then
    \[
    \lim_{N\to \infty} \prod_{m=1}^N (1-K_m) =0
    \]
    and
    \[
    \lim_{N\to \infty} K_N =0.
    \]
    The interpretation of this is that given $N$ measurements, the weight of the $N+1$:th measurement will be small when averaging $\mu^K_N$ with $z_{N+1}$.
\end{lemma}

\begin{proof}
    Since $P^K_i$ goes to zero as $i\to \infty$ by Equation \eqref{eq:limitvarianceestimate} and since
    \[
    P^K_N = \prod_{i=1}^N (1-K_i) \hat{\eta}(\sigma^2_0)
    \]
    it must follow that
    \[
    \prod_{i=1}^\infty (1-K_i) =0.
    \]
    Moreover, since 
    \[
    K_{i} = \frac{\hat{\eta}^{-1}(P^K_{i-1})}{\xi^2 +\hat{\eta}^{-1}(P^K_{i-1})}
    \]
    and since $\lim_{i\to \infty} P^K_{i}$ and $\xi^2> 0$ it follows that $\lim_{i\to\infty }K_i=0$.
\end{proof}

We shall have need of one more lemma before we can finish the proof as we have need for a way of expanding $\dist(p,\mu_m^K)$.

\begin{lemma}
\label{lem:linearization}
    Consider $\mu_K^m$, the filtering mean by repeating $m$ steps of the Extended Kalman filter given in Algorithm \ref{alg:kalmansphere} and let $p= X_0(\omega)$ be the realization of the constant process $X_t=X_0$. Then 
    \begin{multline}
        \log_p(\mu^K_m) \approx K_m\log_p(z_m)+ \sum_{j=1}^{m-1} \left(\prod_{i=1}^j (1-K_{m-i}) \right) K_{m-j} \log_p(z_{m-j})\\
        + \left( \prod_{i=1}^m (1-K_{i}) \right) \log_p(\mu_0)
        \label{eq:linearizationlogfilterstep}
    \end{multline}
    up to first order.
\end{lemma}

\begin{proof}
    Note that
    \[
    K_m\log_{\mu^K_m}(z_m ) + (1-K_m)\log_{\mu^K_m}(\mu^K_{m-1}) =0 
    \]
    and then through linearization it holds that
    \[
    \log_p(\mu^K_m) = K_m \log_p (z_m) + (1-K_m)\log_{p}(\mu^K_{m-1}) + \mathcal{O}(\dist(\mu^K_m,p)^2),
    \]
    by \cite[Theorem 1]{hotz2024central}. Now, Equation \eqref{eq:linearizationlogfilterstep} follows inductively.
\end{proof}

\begin{proof}[Proof of Theorem \ref{thm:convofextkalman}]
Note that
\[
\dist(p,\mu_m^K)= \norm{\log_p(\mu_m^K)}
\]
and then by Lemma \ref{lem:linearization} it holds that
\begin{multline*}
\dist(p,\mu_m^K)\\ \leq \norm{K_m\log_p(z_m)+ \sum_{j=1}^{m-1} \left(\prod_{i=1}^j (1-K_{m-i}) \right) K_{m-j} \log_p(z_{m-j})}\\ + \left( \prod_{i=1}^m (1-K_{i}) \right) \norm{\log_p(\mu_0)}.
\end{multline*}
Firstly, note that the first term satisfy Lindeberg's conditions by \ref{lem:Kalmangains} and thus converge to zero in probability by Lindeberg's central limit theorem \cite{lindeberg1922neue}. Secondly, the second term goes to zero as $m \to \infty$ by Lemma \ref{lem:Kalmangains}.    
\end{proof}

\section{Simulations}
\label{sec:simulations}
We shall demonstrate some experimental verification of Theorem \ref{thm:convofextkalman}. All simulations have been performed using MATLAB. They have all been set up by drawing a point $x_0$ from $ N(\mu_0,\sigma_0^2 I_{\mathrm{M}_{n,k}(\mathbb{K})})$ and then drawing $N$ independent observations from $\pr(Z)$, where 
\[
Z \eqdist N (x_0, \xi^2 I_{\mathrm{M}_{n,k}(\mathbb{K})}),
\]
and then computing $\pr(Z)= Z(Z^*Z)^{-1/2}$. We fix the parameters $A=0\in \mathrm{M}_{n,n}(\mathbb{K})$ and $\nu=0$ in \eqref{eq:filteringSDE} and we fix initial mean $\mu_0= I_{n,k}$. Then the filtering algorithm, see Algorithm \ref{alg:kalmansphere}, is implemented to gain estimates of the mean and scalar variance. The estimated scalar variance $P^K_m$ and the distance squared from $\mu^K_m$ to $x_0$ normalized with the dimension are plotted for different values of measurement errors, different starting variances and for different Stiefel manifolds. Simulations are done for the following selection of Stiefel manifolds; $\St_{4,2}(\mathbb{R})$ in Figure \ref{fig:St42}, $\St_{6,3}(\mathbb{R})$ in Figure \ref{fig:St63}, $\St_{12,3}(\mathbb{R})$ in Figure \ref{fig:St123} and $\St_{15,5}(\mathbb{R})$ in Figure \ref{fig:St155}. One can see that for $\xi^2=0.1$, the algorithm works reasonably well as the distance squared of the filtered mean estimate is close to the estimated Kalman scalar variance. However, already for $\xi^2=0.5$ one can see that there will be some significant additional error propagation due to the linearization of the fundamentally non-linear filtering problem defined in \eqref{eq:filteringSDE}. Moreover, one can also see that this error propagation gets exaggerated when the Stiefel manifolds are of higher dimension. Nonetheless, the filtering algorithm seem to converge, though at a slower speed, even for larger measurement errors.

\begin{figure}[H]
    \centering
    \begin{subfigure}{0.85\textwidth}
        \includegraphics[width=\linewidth]{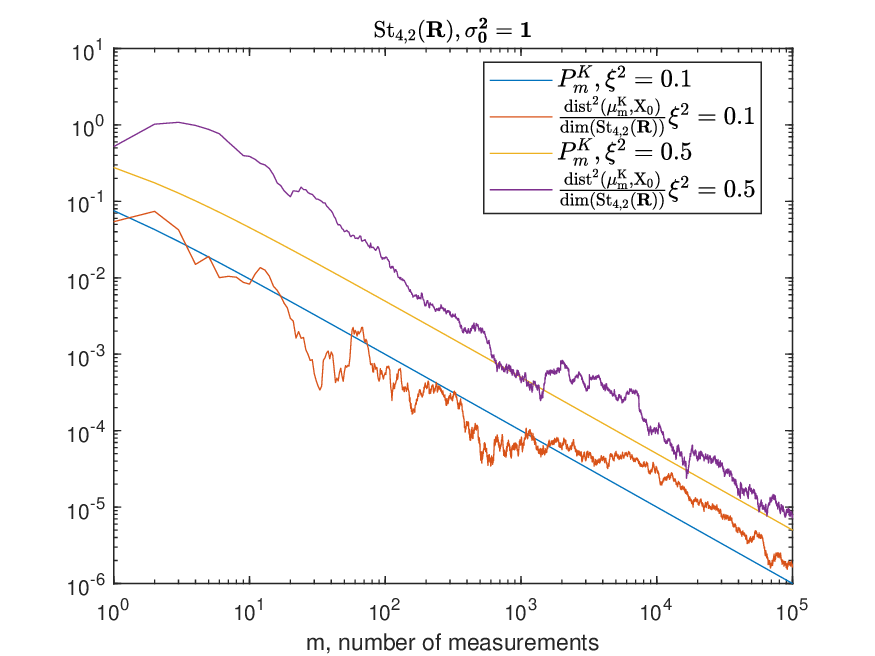}
        \caption{$\sigma_0^2=1$.}
    \end{subfigure}
\end{figure}
\begin{figure}[H]
\ContinuedFloat
    \begin{subfigure}{0.85\textwidth}
        \includegraphics[width=\linewidth]{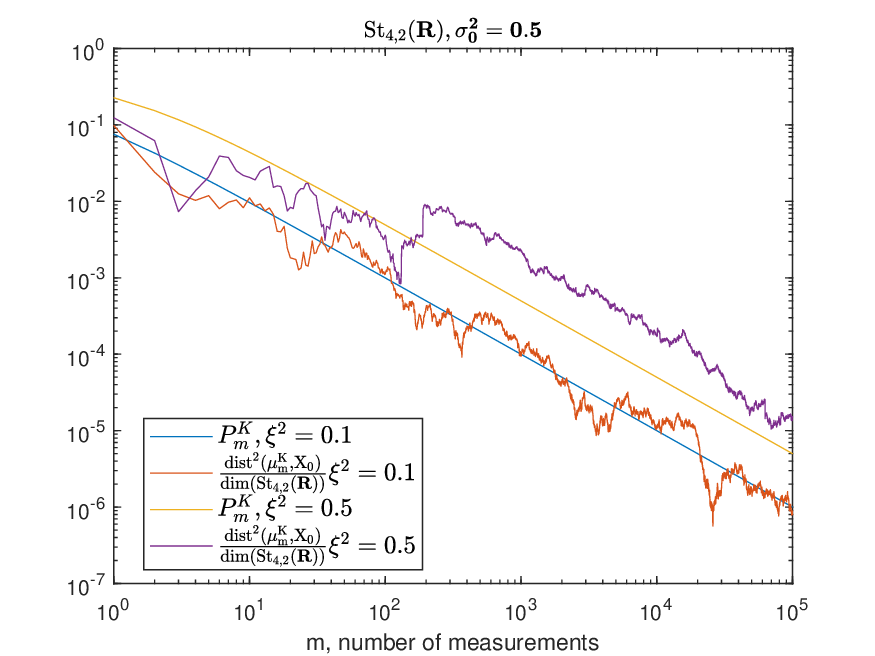}
        \caption{$\sigma_0^2=0.5$.}
    \end{subfigure}
\end{figure}
\begin{figure}[H]
\ContinuedFloat
    \begin{subfigure}{0.85\textwidth}
        \includegraphics[width=\linewidth]{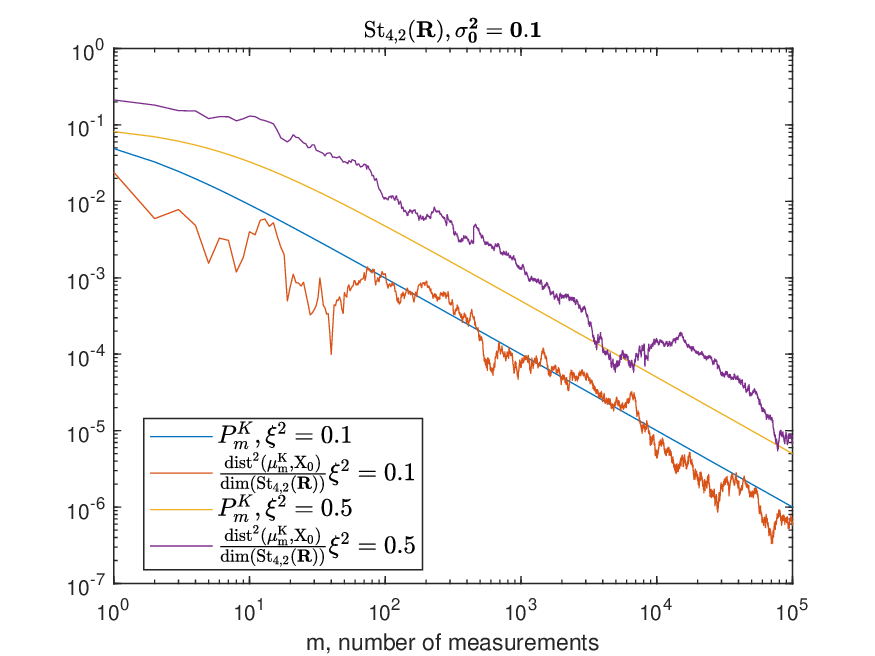}
        \caption{$\sigma_0^2=0.1$.}
    \end{subfigure}
    \caption{ Simulations of the extended Kalman filter on $\St_{4,2}(\mathbb{R})$ given in Algorithm \ref{alg:kalmansphere}. The intrinsic scalar variance and the normalized experimental distance squared from the filtered mean to the true point over $m$ measurements. In each figure one can see both given $\xi^2=0.1$ and $\xi^2=0.5$.}
    \label{fig:St42}
\end{figure}

\begin{figure}[H]
    \centering
    \begin{subfigure}{0.85\textwidth}
        \includegraphics[width=\linewidth]{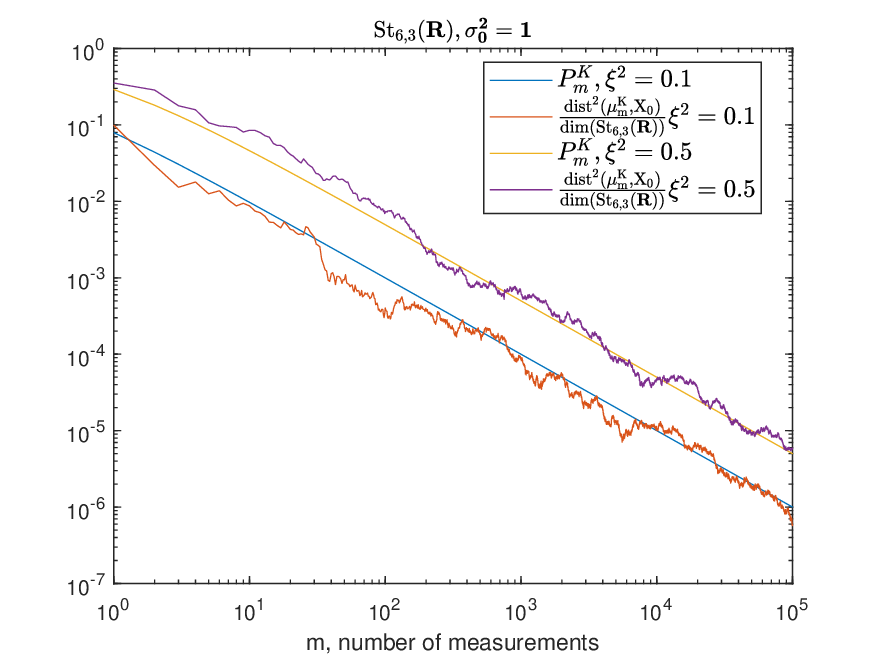}
        \caption{$\sigma_0^2=1$.}
    \end{subfigure}
\end{figure}
\begin{figure}[H]
\ContinuedFloat
    \begin{subfigure}{0.85\textwidth}
        \includegraphics[width=\linewidth]{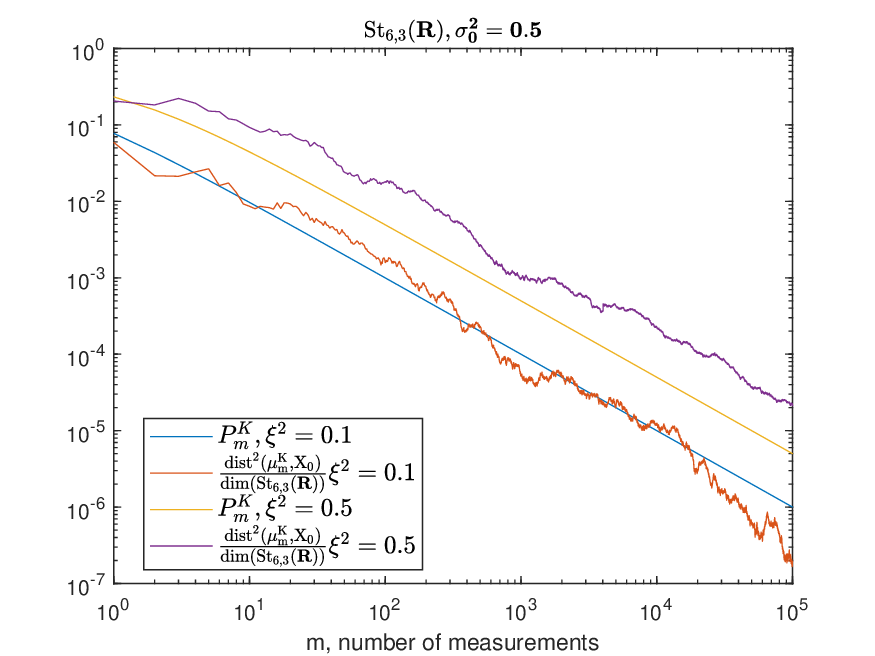}
        \caption{$\sigma_0^2=0.5$.}
    \end{subfigure}
\end{figure}
\begin{figure}[H]
\ContinuedFloat
    \begin{subfigure}{0.85\textwidth}
        \includegraphics[width=\linewidth]{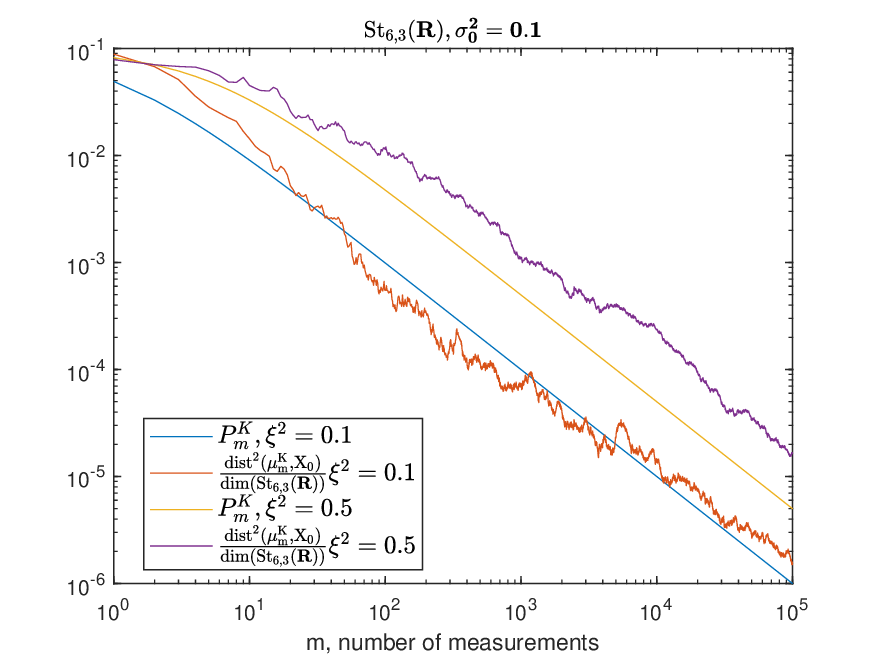}
        \caption{$\sigma_0^2=0.1$.}
    \end{subfigure}
    \caption{ Simulations of the extended Kalman filter on $\St_{6,3}(\mathbb{R})$ given in Algorithm \ref{alg:kalmansphere}. The intrinsic scalar variance and the normalized experimental distance squared from the filtered mean to the true point over $m$ measurements. In each figure one can see both given $\xi^2=0.1$ and $\xi^2=0.5$.}
    \label{fig:St63}
\end{figure}

\begin{figure}[H]
    \centering
    \begin{subfigure}{0.85\textwidth}
        \includegraphics[width=\linewidth]{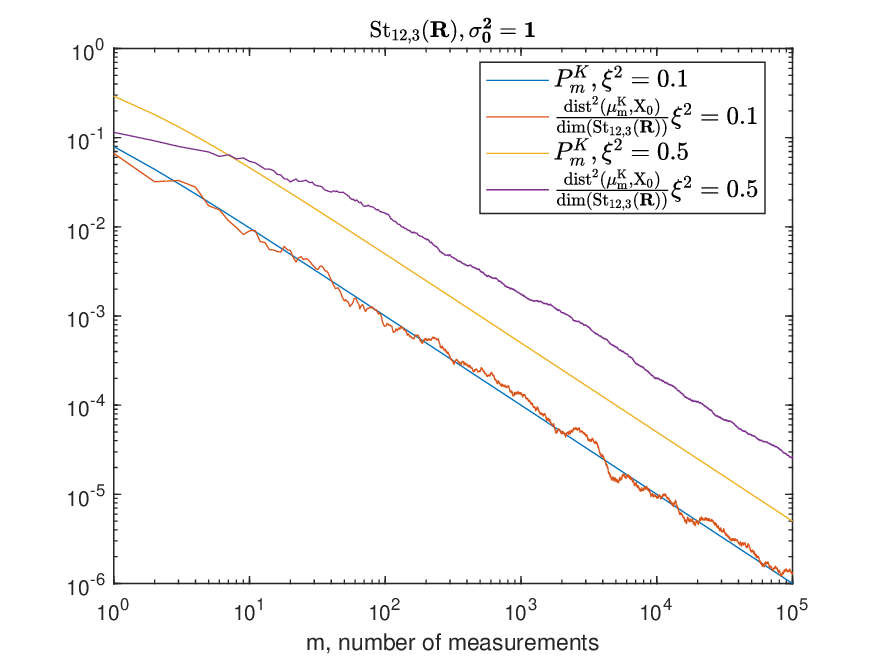}
        \caption{$\sigma_0^2=1$.}
    \end{subfigure}
\end{figure}
\begin{figure}[H]
\ContinuedFloat
    \begin{subfigure}{0.85\textwidth}
        \includegraphics[width=\linewidth]{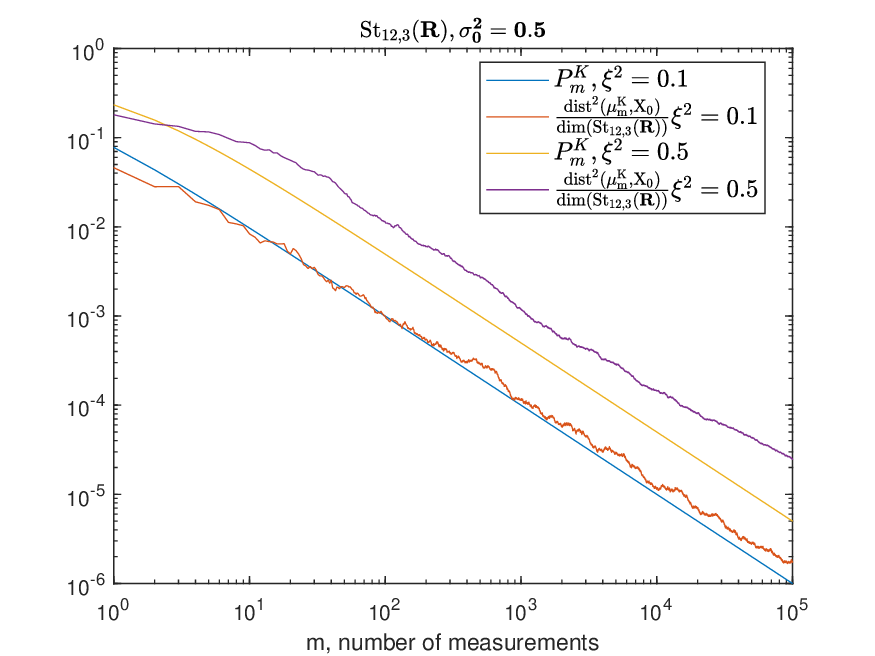}
        \caption{$\sigma_0^2=0.5$.}
    \end{subfigure}
\end{figure}
\begin{figure}[H]
\ContinuedFloat
    \begin{subfigure}{0.85\textwidth}
        \includegraphics[width=\linewidth]{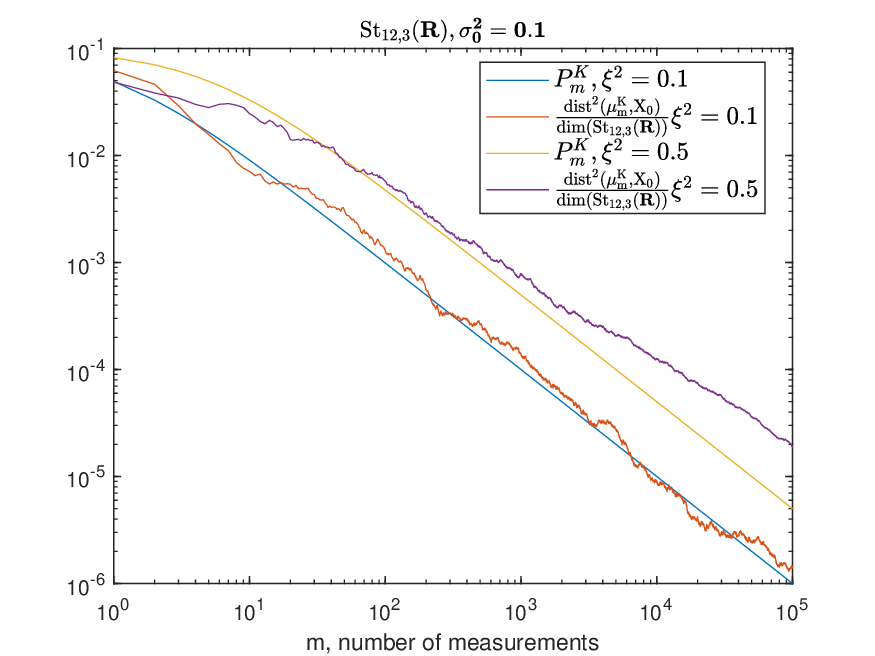}
        \caption{ $\sigma_0^2=0.1$.}
    \end{subfigure}
    \caption{Simulations of the extended Kalman filter on $\St_{12,3}(\mathbb{R})$ given in Algorithm \ref{alg:kalmansphere}. The intrinsic scalar variance and the normalized experimental distance squared from the filtered mean to the true point over $m$ measurements. In each figure one can see both given $\xi^2=0.1$ and $\xi^2=0.5$.}
    \label{fig:St123}
\end{figure}

\begin{figure}[H]
    \centering
    \begin{subfigure}{0.85\textwidth}
        \includegraphics[width=\linewidth]{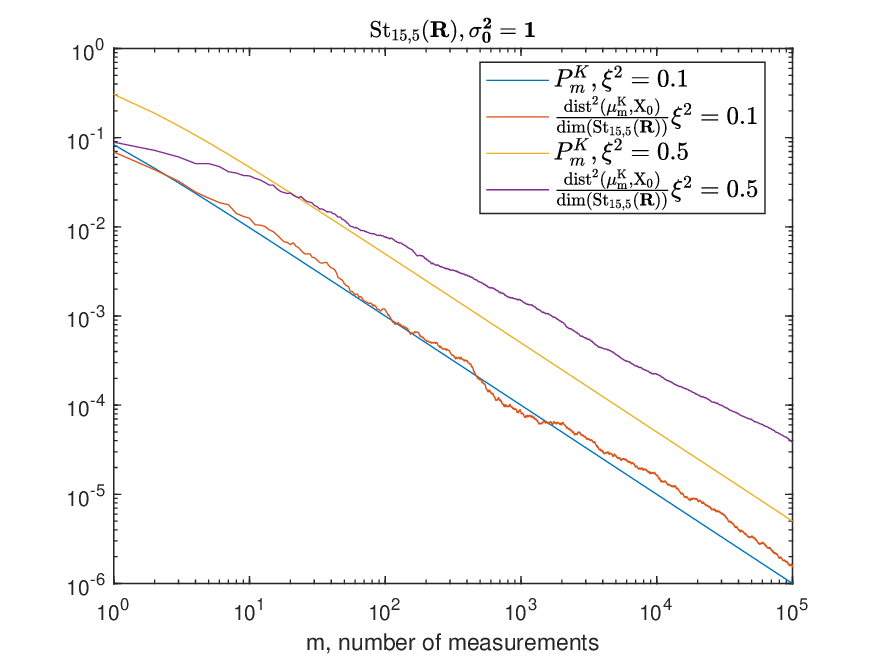}
        \caption{\small $\sigma_0^2=1$.}
    \end{subfigure}
\end{figure}
\begin{figure}[H]
\ContinuedFloat
    \begin{subfigure}{0.85\textwidth}
        \includegraphics[width=\linewidth]{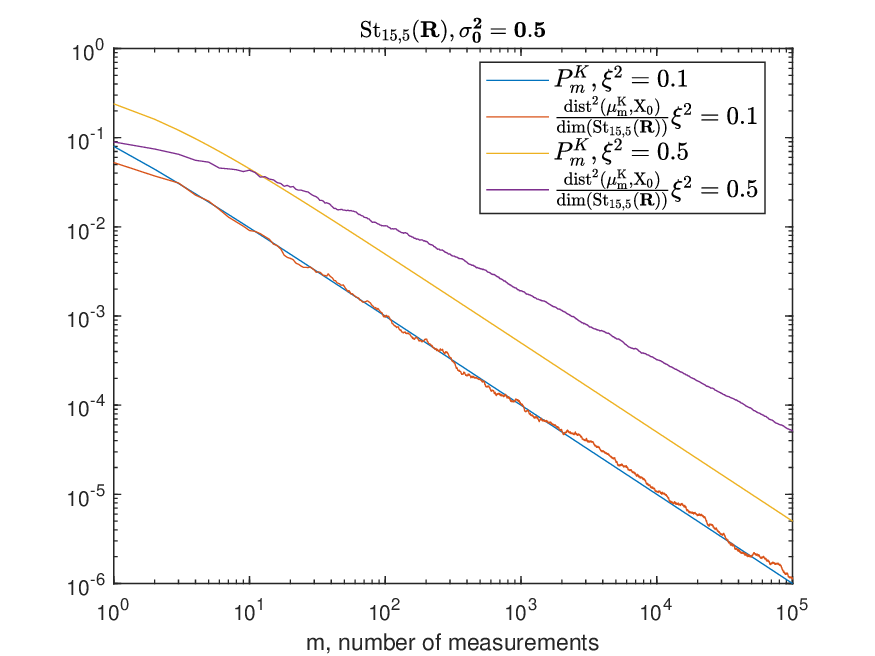}
        \caption{ $\sigma_0^2=0.5$.}
    \end{subfigure}
\end{figure}
\begin{figure}[H]
\ContinuedFloat
    \begin{subfigure}{0.85\textwidth}
        \includegraphics[width=\linewidth]{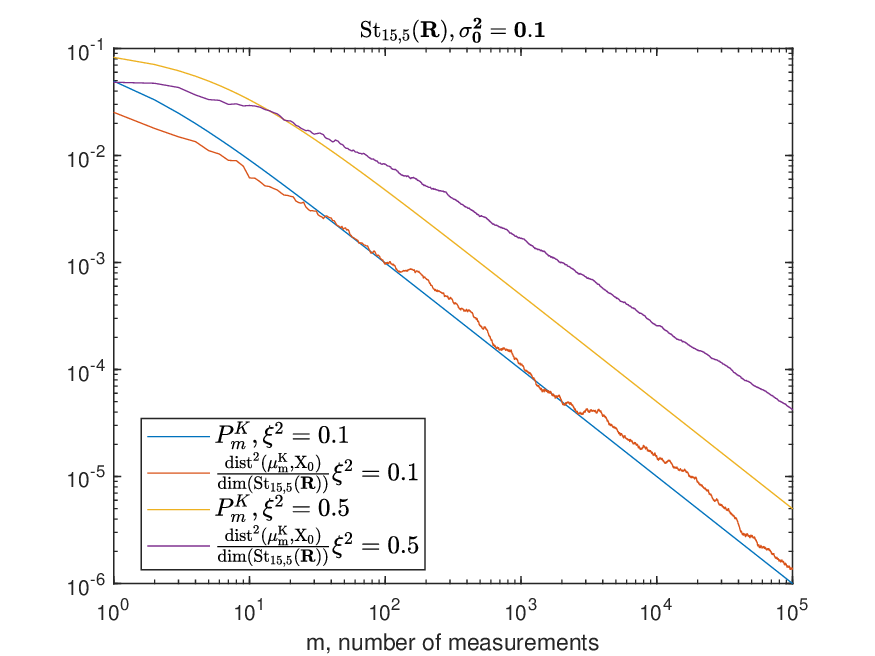}
        \caption{  $\sigma_0^2=0.1$.}
    \end{subfigure}
    \caption{\small Simulations of the extended Kalman filter on $\St_{15,5}(\mathbb{R})$ given in Algorithm \ref{alg:kalmansphere}. The intrinsic scalar variance and the normalized experimental distance squared from the filtered mean to the true point over $m$ measurements. In each figure one can see both given $\xi^2=0.1$ and $\xi^2=0.5$.}
    \label{fig:St155}
\end{figure}

\printbibliography

\end{document}